\documentclass{tran-l}

\usepackage[centertags]{amsmath}
\usepackage{amsfonts}
\usepackage{amssymb}
\usepackage{amsthm}
\usepackage{amsmath}
\usepackage{newlfont}
\usepackage{lscape}
\usepackage{graphicx}
\newlength{\defbaselineskip}
\newcommand{\setlinespacing}[1]
          {\setlength{\baselineskip}{#1 \defbaselineskip}}

 \newcommand{\lphi}{\varphi}
 \newcommand{\N}{\mathbb{N}}
 \newcommand{\bth}{\begin{thm}}
 \newcommand{\beq}{\begin{eqnarray*}}
 \newcommand{\eeq}{\end{eqnarray*}}
 \DeclareRobustCommand{\rs}{\genfrac\{\}{0pt}{}}

\newtheorem{thm}{Theorem}[section]
\newtheorem{lem}[thm]{Lemma}
\newtheorem{prop}[thm]{Proposition}

\begin{document}

\title[Periodicity of residual $r$-Fubini sequences]{On the periodicity problem of\\ residual $r$-Fubini sequences}

\author{Amir Abbas Asgari}
\address{National Organization for Development of Exceptional Talents (NODET), Tehran, Iran}
\email{asgari@helli.ir}

\author{Majid Jahangiri}
\address{School of Mathematics, Institute for Research in Fundamental Sciences (IPM), Tehran, Iran}
\email{jahangiri@ipm.ir}


\date{}


\dedicatory{}

\begin{abstract}
For any positive integer $r$, the $r$-Fubini number with parameter $n$, denoted by $F_{n,r}$, is equal to the number of ways that the elements of a set with $n+r$ elements can be weak ordered such that the $r$ least elements are in distinct orders. In this article we focus on the sequence of residues of the $r$-Fubini numbers modulo a positive integer $s$ and show that this sequence is periodic and then, exhibit how to calculate its period length.  As an extra result, an explicit formula for the $r$-Stirling numbers is obtained which is frequently used in calculations.
\end{abstract}

\maketitle

%----------------------------------------------------------

\section{Introduction}
\textit{The Fubini numbers} (also known as the \textit{Ordered Bell numbers}) form an integer sequence in which the $n$th term counts the number of weak orderings of a set with $n$ elements. Weak ordering means that the elements can be ordered, allowing ties. A. Cayley studied the Fubini numbers as the number of a certain kind of trees with $n+1$ terminal nodes [2]. The Fubini numbers can also be defined as the summation of the \textit{Stirling numbers of the second kind}. The Stirling number of the second kind which is denoted by $\rs{n}{k}$, counts the number of partitions of $n$ elements into $k$ non-empty subsets.  The sequence of residues of the Fubini numbers modulo a positive integer $s$ is pointed out by Bjorn Poonen. He showed that this sequence is periodic and calculated the period length for each positive integer $s$ [5].

\textit{The $r$-Stirling numbers of the second kind} are defined as an extension to the Stirling numbers of the second kind, and similarly, an $r$-Fubini number is defined as the number of ways which the elements of a set with $n+r$ elements can be weak ordered such that the first $r$ elements are in distinct places. One can study the same problem of periodicity of residual sequence in case of the $r$-Fubini numbers. I. Mezo investigated this problem for $s=10$ [4]. In this article, $\omega(A_{r,s})$, the period of the $r$-Fubini numbers modulo any positive integer $s \in \mathbb{N}$ is computed. Based on the Fundamental Theorem of Arithmetic, $\omega(A_{r,p})$ is calculated for powers of odd primes $p^m$. The cases $s=2^m$ are studied separately. Therefore if $s=2^m p_{1}^{m_{1}} p_{1}^{m_{1}} \ldots p_{k}^{m_{k}}$ is the prime factorization, then the $\omega(A_{r,s})$ is equal to the least common multiple (LCM) of $\omega(A_{r,p_i^{m_i}})$'s and $\omega(A_{r,2^m})$, for $i=1,2,\ldots,k$.

The preliminaries are presented in Section \ref{Basic concepts} and in Sections \ref{primes} and \ref{two} the length of the periods are computed in the case of odd prime powers and the 2 powers, respectively. The last section contains the final theorem which presents the conclusion of this article.

\section{Basic Concepts} \label{Basic concepts}

The Stirling number of the second kind with the parameters $n$ and $k$ counts the number of ways that the set $A=\{1, 2,\ldots, n\}$ with $n$ elements can be partitioned into $k$ non-empty subsets. If we want that the first $r$ elements of $A$ are in distinct subsets, the number of ways to do so is the $r$-Stirling number of the second kind with parameters $n$ and $k$, which is denoted by $\rs{n}{k}_{r}$ (so it is clear that $n \geq k \geq r$). Fubini numbers are defined as follows [4]
\begin{align*}
F_{n}=\sum_{k=0}^{n} k!\rs{n}{k}.
\end{align*}
In a similar way we can define the $r$-Fubini numbers as the number of ways which the elements of $A$ can be weak ordered such that the elements $\{1,2,\ldots,r\}$ are in distinct ranks. These numbers are denoted by $F_{n,r}$ and are evaluated by
\begin{align*}
F_{n,r}=\sum_{k=0}^{n} (k+r)!\rs{n+r}{k+r}_{r}.
\end{align*}
There are simple relations and formulae about $\rs{n}{k}_{r}$ which are listed below. One can find a proof of them in [4], [1] and [3, \S 4].
\begin{align}
\rs{n}{m}_{r}&= \rs{n}{m}_{r-1}-(r-1)\rs{n-1}{m}_{r-1}, 1 \leq r \leq n \label{1}\\
\rs{n}{m}_{1}&= \rs{n}{m}\label{2}\\
\rs{n+r}{r}_{r}&= r^n\label{3}\\
\rs{n+r}{r+1}_{r}&= (r+1)^{n}-r^n\label{4}\\
\rs{n}{m}&= \frac{1}{m!}\sum_{j=1}^{m}(-1)^{m-j}\binom{m}{j}j^n\label{5}.
\end{align}

In addition to above recurrence relations of the $r$-Stirling numbers of the second kind, a direct way to compute these numbers is given in the next theorem.
\bth
For $n, m \in\mathbb N$ and $r \leq m \leq n$, the $r$-Stirling number of the second kind with the parameters $n$ and $m$ is
\begin{align*}
\rs{n}{m}_{r}=\frac{1}{m!}\sum_{j=r}^{m}(-1)^{m-j}\binom{m}{j}j^{n-(r-1)} \left(\frac{(j-1)!}{(j-r)!}\right).
\end{align*}
\end{thm}

\begin{proof}
We prove it by the induction on $r$. For $r=2$, the relations (\ref{1}) and (\ref{5}) result
\begin{align*}
\rs{n}{m}_{2}=&\rs{n}{m}-\rs{n-1}{m}\\
=&\frac{1}{m!}\left(\sum_{j=1}^{m}(-1)^{m-j}\binom{m}{j}j^{n}-\sum_{j=1}^{m}(-1)^{m-j}\binom{m}{j}j^{n-1}\right)\\
=&\frac{1}{m!}\sum_{j=2}^{m}(-1)^{m-j}\binom{m}{j}j^{n-1}(j-1)
\end{align*}
Assume that
\begin{eqnarray}
\rs{n}{m}_{r}=\frac{1}{m!}\sum_{j=r}^{m}(-1)^{m-j}\binom{m}{j}j^{n-(r-1)} \left(\frac{(j-1)!}{(j-r)!}\right).
\end{eqnarray}
For $\rs{n}{m}_{r+1}$, use (\ref{1}) to conclude that

\begin{align*}
\rs{n}{m}_{r+1} & =\rs{n}{m}_{r}-r\rs{n-1}{m}_{r}\\
&=\frac{1}{m!}\sum_{j=r}^{m}(-1)^{m-j}\binom{m}{j}j^{n-(r-1)} \left(\frac{(j-1)!}{(j-r)!}\right)\\
&-r \left(\frac{1}{m!}\right) \sum_{j=r}^{m}(-1)^{m-j}\binom{m}{j}j^{n-1-(r-1)} \left(\frac{(j-1)!}{(j-r)!}\right)\\
&=\frac{1}{m!}\sum_{j=r}^{m}(-1)^{m-j}\binom{m}{j} j^{n-r} (j-r) \left(\frac{(j-1)!}{(j-r)!}\right)\\
&=\frac{1}{m!}\sum_{j=r+1}^{m}(-1)^{m-j}\binom{m}{j} j^{n-r} \left(\frac{(j-1)!}{(j-(r+1))!}\right).
\end{align*}
\end{proof}

By $\varphi(n)$ we indicate the number of integer numbers less than $n$ and co-prime to it. It is known as Euler's totient function. The value of $\varphi(n)$ can be computed via the following relation [3]
\begin{align*}
\varphi(n)=n\prod_{p \mid n}(1- \frac{1}{p})
\end{align*}

\section{The $r$-Fubini residues modulo prime powers} \label{primes}
Let $p$ be a prime number greater than 2 and $m$ be a positive integer. If $\{F_{n,r}\}$ denotes the $r$-Fubini numbers for a fixed positive integer $r$, we indicate by $A_{r,q}=\{F_{n,r}$ (mod $q$)$\}$, for $n\in \mathbb{N}$, the sequence of residues of the $r$-Fubini numbers modulo the positive integer $q$. In this section we try to compute the period length of the sequence $A_{r,q}$ when $q=p^m$. We denote this length by $\omega(A_{r,q})$.

\begin{prop}
Let $p$ be an odd prime and let $q=p^m$, $m\in\N$. If $q \leq r$, then $\omega(A_{r,q})=1$
\end{prop}
\begin{proof}
The proof is very simple. Since $p \leq r$, we can deduce that $p \mid (k+r)!$, for $k \geq 0$, and by the relation $F_{n,r}=\sum_{k=0}^{n}(k+r)!\rs{n+r}{k+r}_{r}$, we have $p \mid F_{n,r}$. Therefore $\omega(A_{r,p})=1$.
\end{proof}
As pointed out in above proposition, the cases in which the prime power factor of $s$ is less than or equal to $r$ have the period length 1, so it is sufficient to investigate the period length in the cases of $q>r$.

\begin{lem}\label{firstlem}
Let $p$ be an odd prime and $r,m\in \mathbb{N}$ with $p \geq r+1$. Then
\begin{align*}
p^{m}-r \geq m.
\end{align*}
\end{lem}

\begin{proof}
For $m=1$ the result is obvious. Suppose the inequality holds for any $m\geq 2$. Since $p(p+m) > 2(p+m) > 2p+m$, we have
\begin{eqnarray}
p^2+pm-p\geq p+m\label{6}.
\end{eqnarray}
Since $p-1 \geq r$, the induction hypothesis can be reformulated to $p^m\geq p-1+m$. Multiplication by $p$ results $p^{m+1}\geq p^2+pm-p$. By (\ref{6}) we have $p^{m+1}\geq p+(m+1)-1$, Q.E.D.
\end{proof}

\bth
Let $p$ be an odd prime and $q=p^m$. After the $(m-1)$th term the sequence $A_{r,q}$ has a period with length $\omega(A_{r,q})=\lphi(q)$.
\end{thm}

\begin{proof}
If $n \geq q-r-1$ we can write
\begin{align*}
F_{n+\lphi(q),r}- & F_{n,r}= \sum_{k=0}^{n+\lphi(q)}(k+r)!\rs{n+\lphi(q)+r}{k+r}_{r}-\sum_{k=0}^{n}(k+r)!\rs{n+r}{k+r}_{r}\\
\equiv & \sum_{k=0}^{q-r-1}(k+r)! \left(\rs{n+\lphi(q)+r}{k+r}_{r}-\rs{n+r}{k+r}_{r}\right) \hbox{(mod  $q$)}\\
\equiv & \sum_{k=0}^{q-r-1} \sum_{j=r}^{k+r}(-1)^{k+r-j} \binom{k+r}{j}j^{n+1} \left(\frac{(j-1)!}{(j-r)!}\right)(j^{\lphi (q)}-1) \hbox{ (mod $q$)}.
\end{align*}
If $j=cp$, $c\in\N$,  then $j^{n+1}=(cp)^{q-r+h}$, for some $h\geq 0$, so from Lemma \ref{firstlem} it follows that $j^{n+1}\equiv 0$ (mod $q$). If $(j,q)=1$, by Euler's Theorem $j^{\lphi(q)}-1\equiv 0$(mod $q$), so the right hand side of the above congruence relation vanished and we have
\begin{align}\label{case1}
F_{n+\lphi(q),r} \equiv & F_{n,r} \hbox{(mod $q$), for $n \geq q-r-1$}.
\end{align}
If $m-1 \leq n<q-r-1$ then
\begin{align*}
F_{n+\lphi(q),r}-  F_{n,r} \equiv  \sum_{k=0}^{q-r-1} & (k+r)! \left(\rs{n+\lphi(q)+r}{k+r}_{r}-\rs{n+r}{k+r}_{r}\right)\\
- & \sum_{k=n+\lphi(q)+1}^{q-r-1}(k+r)!\rs {n+\lphi(q)+r}{k+r}_{r} \\
+ & \sum_{k=n+1}^{q-r-1}(k+r)!\rs {n+r}{k+r}_{r} \hbox{(mod $q$)}\\
\end{align*}
\begin{align*}
\equiv \sum_{k=0}^{q-r-1} & \sum_{j=r}^{k+r}(-1)^{k+r-j} \binom{k+r}{j}j^{n+1} \left(\frac{(j-1)!}{(j-r)!}\right) (j^{\lphi (q)}-1)\\
- & \sum_{k=n+\lphi(q)+1}^{q-r-1}(k+r)!\rs {n+\lphi(q)+r}{k+r}_{r} \\
+ & \sum_{k=n+1}^{q-r-1}(k+r)!\rs {n+r}{k+r}_{r} \hbox{(mod $q$)}.
\end{align*}
Since $n\geq m-1$, in the indices where $j=c p$, $c\in\N$, we have $j^{n+1}=(cp)^{m+h}$, for some $h \geq 0$, and it deduced that $j^{n+1} \equiv 0$ (mod $q$). When $(j,q)=1$, again by Euler's Theorem $j^{\lphi(q)}-1\equiv 0$(mod $q$). In the sums $\sum_{k=n+1}^{q-r-1}(k+r)!\rs {n+r}{k+r}_{r}$ and $\sum_{k=n+\lphi(q)+1}^{q-r-1}(k+r)!\rs {n+\lphi(q)+r}{k+r}_{r}$ the upper parameter of the $r$-Stirling number is less than the lower one and therefore these two sums are equal to zero. So
\begin{align*}
F_{n+\lphi(q),r}-F_{n,r} \equiv & \sum_{k=0}^{q-r-1} \sum_{j=r}^{k+r}(-1)^{k+r-j} \binom{k+r}{j}j^{n+1}\left( \frac{(j-1)!}{(j-r)!} \right)\\
\times & (j^{\lphi (q)}-1) \equiv 0 \hbox{(mod $q$)},
\end{align*}
and therefore
\begin{align}\label{case2}
F_{n+\lphi (q),r} \equiv F_{n,r} \hbox{(mod $q$), for $m-1 \leq n <q-r-1$}.
\end{align}
Combining results (\ref{case1}) and (\ref{case2}) gives $F_{n+\lphi(q),r} \equiv F_{n,r}$ (mod $q$), for $n \geq m-1$.
\end{proof}

\section{The $r$-Fubini residues modulo powers of 2} \label{two}
Similar to many computations in number theory, the case of $p=2$ has its own difficulties which needs special manipulations. In the case of powers of 2, initially we calculate the residues of 2-Fubini numbers and then use the results in the case of the $r$-Fubini numbers. We classify the sequence of remainders of 2-Fubini numbers modulo $2^m$, $m \geq 7$, in Theorem \ref{thm2fubini128}  and then, work on remainders of the $r$-Fubini numbers modulo $2^{m}$, $m \geq 7$ in Theorem \ref{thmrfubini128}. The special cases will be proved in Theorems \ref{thm2fubini8to64}, \ref{thmrfubini4,8} and \ref{thmrfubini8to64}. The trivial cases in which $2^m\leq r$ with period length 1 are omitted.

\begin{thm} \label{thm2fubini8to64}
If $3 \leq m \leq 6$, then after the $(m-1)$th term the sequence $A_{2,2^m}$ has a period with length $\omega(A_{2,2^m})=2$.
\end{thm}

\begin{proof}
By using the formula $F_{n,2}=\sum_{k=0}^{n}(k+2)!\rs{n+2}{k+2}_{2}$ we prove that $F_{n+2,2}-F_{n,2} \equiv 0$ (mod $2^6$) then it is concluded that $F_{n+2,2}-F_{n,2} \equiv 0$ (mod $2^m$), $3\leq m \leq 5$.
\begin{align*}
F_{n+2,2}-F_{n,2}= & \sum_{k=0}^{n+2}(k+2)!\rs{n+4}{k+2}_{2}-\sum_{k=0}^{n}(k+2)!\rs{n+2}{k+2}_{2}\\
\equiv & \sum_{k=0}^{5}(k+2)! \left(\rs{n+4}{k+2}_{2}-\rs{n+2}{k+2}_{2}\right) \hbox{(mod $2^6$)}\\
\equiv & \sum_{k=0}^{5}\sum_{j=2}^{k+2}(-1)^{k+2-j}\binom{k+2}{j} j^{n+1} (j^2-1) (j-1) \hbox{ (mod $2^6$)}.
\end{align*}
$m=6$ implies that $n \geq 5$, so if $j$ is even, then $j^{n+1}=(2c)^{6+h}$, for some $h \geq 0$ and therefore $64 \mid j^{n+1}$. For odd $j$'s, $(j,64)=1$ so by Euler's Theorem we have $j^{32} \equiv 1$ (mod 64) and therefore $j^{n+1+32} \equiv j^{n+1}$ (mod 64). This implies
\begin{align*}
F_{n+2,2}-F_{n,2} \equiv \sum_{k=0}^{5}\sum_{l=1}^{\lfloor (k+1)/2 \rfloor} & (-1)^{k+2-(2l+1)}\binom{k+2}{2l+1}  (2l+1)^{n+1}\\
& \times \left((2l+1)^2-1)\right) \times 2l \hbox{ (mod 64)}\\
\equiv  16\sum_{k=0}^{5}(-1)^{k+1} & \left(\sum_{l=1}^{\lfloor (k+1)/2  \rfloor}\binom{k+2}{2l+1}  (2l+1)^{n+1} \left(\frac{l(l+1)}{2}\right) l \right).
\end{align*}
Enumerating the last summation for $2 \leq n \leq 33$ shows that it is divisible by 64 and because of periodicity of remainders of $j^{n+1}$ modulo 64, the result follows.
\end{proof}

Now we present the following lemma analogues to Lemma \ref{firstlem} in previous section.
\begin{lem}\label{secondlem}
If $m \in \mathbb{N}$ and $m > 1$, then $2^m-2 \geq m$.
\end{lem}

\begin{proof}
For m=2 the result is obvious. If we assume that
\begin{eqnarray}
2^m-2 \geq m\label{8}
\end{eqnarray}
then multiplication by 2 gives $2^{m+1}-4 \geq 2m$. Since $2^{m+1} \geq 2m+4 > m+3$, we have $2^{m+1} \geq m+3$ and so $2^{m+1}-2 \geq m+1$.
\end{proof}
The following lemma provides a simple but essential relation used in the next theorem. Its proof is provided in Appendix A.
\begin{lem}\label{appendix}
For $m\geq 7$ and $5 \leq i \leq 2^{m-6}$ we have $2^{m-6}-i \mid 2^{i-5}\binom{2^{m-6}-1}{i}$.
\end{lem}

\begin{thm} \label{thm2fubini128}
If $m \geq 7$, after the $(m-1)$th term, the sequence $A_{2,2^m}$ has a period with length $\omega(A_{2,2^m})=2^{m-6}$.
\end{thm}

\begin{proof}
In the case of $n \geq 2^m-3$, from Lemma \ref{secondlem} we can deduce that $n \geq 2^m-3 \geq m-1
$. So we have
\begin{align*}
F_{n+2^{m-6},2}- & F_{n,2} \equiv \sum_{k=0}^{n+2^{m-6}}(k+2)!\rs{n+2^{m-6}+2}{k+2}_{2}-\sum_{k=0}^{n}(k+2)!\rs{n+2}{k+2}_{2} \\
\equiv & \sum_{k=0}^{2^m-3}(k+2)!\left(\rs{n+2^{m-6}+2}{k+2}_{2}-\rs{n+2}{k+2}_{2}\right) \hbox{(mod $2^m$)}\\
\equiv & \sum_{k=0}^{2^m-3}\sum_{j=2}^{k+2}(-1)^{k+2-j}\binom{k+2}{j}j^{n+1} (j^{2^{m-6}}-1) (j-1) \hbox{ (mod $2^m$)}.
\end{align*}
When $j$ is even, then $j^{n+1}=(2c)^{2^m-2+h}$, for some $h \geq 0$. So by Lemma \ref{secondlem}, $2^m \mid j^{n+1}$. For odd $j$'s we have

\begin{align*}
F_{n+2^{m-6},2}- & F_{n,2} \equiv \sum_{k=0}^{2^m-3}\sum_{l=1}^{\lfloor (k+1)/2 \rfloor}(-1)^{k+2-(2l+1)}\binom{k+2}{2l+1}(2l+1)^{n+1}\\
\times & ((2l+1)^{2^{m-6}}-1) \times 2l \hbox{ (mod $2^m$)}\\
\equiv & 2^{m-4}\sum_{k=0}^{2^m-3}(-1)^{k+1}  \sum_{l=1}^{\lfloor (k+1)/2 \rfloor} \binom{k+2}{2l+1}(2l+1)^{n+1}\\
\times & \left(\frac{(2l+1)^{2^{m-6}}-1}{2^{m-5}}\right) l \hbox{ (mod $2^m$)}\\
\equiv & 2^{m-4}\sum_{k=0}^{2^m-3}(-1)^{k+1} \sum_{l=1}^{\lfloor (k+1)/2 \rfloor} \binom{k+2}{2l+1}(2l+1)^{n+1}\\
\times & \sum_{i=1}^{2^{m-6}}l^i2^{i-1} \left(\frac{(2^{m-6}-1)!}{i!(2^{m-6}-i)!}\right) l \hbox{ (mod  $2^m$)}.
\end{align*}
Last expression contains $m-4$ factors of 2, so it is sufficient to prove the last summation is divisible by 16. We denote this summation by $\mathcal{S}$. Simplify the summation $\sum_{i=1}^{2^{m-6}}l^i2^{i-1}\frac{(2^{m-6}-1)!}{i!(2^{m-6}-i)!}$ and use the Lemma \ref{appendix} gives
\begin{align*}
\sum_{i=1}^{2^{m-6}} l^i 2^{i-1} & \left(\frac{(2^{m-6}-1)!}{i!(2^{m-6}-i)!}\right) \equiv \sum_{i=1}^{4}l^i2^{i-1} \left(\frac{(2^{m-6}-1)!}{i!(2^{m-6}-i)!}\right) \hbox{(mod 16)}\\
\equiv  l+ & l^2(2^{m-6}-1)+\frac{l^3\times 2(2^{m-6}-1)(2^{m-6}-2)}{3}\\
+ & \frac{l^4(2^{m-6}-1)(2^{m-6}-2)(2^{m-6}-3)}{3} \hbox{(mod  16)}.
\end{align*}
Assume $m \geq 10$ (the case $7 \leq m \leq 9$ is proceeded at last). So $16 \mid 2^{m-6}$. Let $3a= 2(2^{m-6}-1)(2^{m-6}-2)$ and $3b= (2^{m-6}-1)(2^{m-6}-2)(2^{m-6}-3)$. Then $3a \equiv 4$(mod 16) and $3b \equiv -6$(mod 16). Therefore $a \equiv -4$(mod 16) and $b \equiv -2$(mod 16). So the proof continues as follows
\begin{align*}
\mathcal S & \overset{16}{\equiv} \sum_{k=0}^{2^m-3}(-1)^{k+1}\left(\sum_{l=1}^{\lfloor (k+1)/2 \rfloor} \binom{k+2}{2l+1}(2l+1)^{n+1} (l-l^2-4l^3-2l^4) l \right) \\
\mathcal S & \overset{8}{\equiv} \sum_{k=0}^{2^m-3}(-1)^{k+1} \sum_{l=1}^{\lfloor (k+1)/2 \rfloor} \binom{k+2}{2l+1}(2l+1)^{n+1} \left( \frac{l(l+1)}{2} \right) (-2l^2-2l+1) l.
\end{align*}
Let $P(l)$ and $A(k,r,n)$ be the remainder of $\frac{1}{2}(2l+1)^{n+1} (l(l+1)) (-2l^2-2l+1) l$ and $\sum_{l=-\infty}^{\infty}\binom{k+2}{2l+r}P(l)$ divided by 8, respectively. By Pascal identity, we have $\binom{k+2}{2l+r}=\binom{k+1}{2l+r}+\binom{k+1}{2l+r-1}$ and therefore
\begin{align*}
\sum_{l=-\infty}^{\infty}\binom{k+2}{2l+r}P(l)=\sum_{l=-\infty}^{\infty}\binom{k+1}{2l+r}P(l)+ \sum_{l=-\infty}^{\infty}\binom{k+1}{2l+r-1}P(l),
\end{align*}
so
\begin{eqnarray}\label{11}
A(k,r,n)=A(k-1,r,n)+A(k-1,r-1,n).
\end{eqnarray}
We can write
\begin{align*}
A(k,r+32,n)\equiv \sum_{l=-\infty}^{\infty}\binom{k+2}{2l+r+32}P(l). \hbox{ (mod 8)}
\end{align*}
The sequence $\left(P(l)\right)_{l=-\infty}^{\infty}$ has period 16, so $P(l+16)=P(l)$. Set $l'=l+16$, then
\begin{eqnarray}\label{12}
A(k,r+32,n)\equiv\sum_{l'=-\infty}^{\infty}\binom{k+2}{2l'+r}P(l') \equiv A(k,r,n) \hbox{ (mod 8)}.
\end{eqnarray}
Since $(2l+1,16)=1$, the Euler's Theorem implies $(2l+1)^{8} \equiv 1 $ (mod 16) and therefore $(2l+1)^{n+1+8} \equiv (2l+1)^{n+1} $ (mod 16). $A(6,r,n)$ vanishes for $1\leq r \leq 32$ and $9 \leq n \leq 24$, by enumeration, then by (\ref{11}) and (\ref{12}), we deduce that
\begin{eqnarray} \label{13}
A(k,r,n)=0, \hbox {  for $k\geq 6$. }
\end{eqnarray}
Therefore
\begin{align*}
A(k,1,n) & \equiv \sum_{l=-\infty}^{\infty}\binom{k+2}{2l+1}(2l+1)^{n+1} \left(\frac{l(l+1)}{2}\right)  (-2l^2-2l+1) l \hbox{ (mod 8)}\\
& \equiv \sum_{l=1}^{\lfloor (k+1)/2 \rfloor}\binom{k+2}{2l+1}(2l+1)^{n+1} \left(\frac{l(l+1)}{2}\right) (-2l^2-2l+1) l \\
& \equiv  0 \hbox{ (mod 8)},
\end{align*}
for $k\geq 6$. If $1 \leq k \leq 5$, $9 \leq n \leq 24$ and $1 \leq r \leq 32$ we have $\sum_{k=1}^{5}(-1)^{k+1}A(k,r,n) \equiv 0$ (mod  8). The period length of $A(k,r,n)$ with respect to $r$ and $n$ implies that \\ $\sum_{k=1}^{5}(-1)^{k+1}A(k,1,n) \equiv 0$ (mod 8), for $n \geq 9$. Combining this with (\ref{13}) we have
\begin{align*}
\mathcal S\equiv \sum_{k=1}^{2^m-3}(-1)^{k+1}A(k,1,n)\equiv 0 \hbox{ (mod 8), for $n \geq 0$.}
\end{align*}
So the result follows in the case of $n\geq 2^m-3$. If $m-1 \leq n <2^m-3$ we can write
\begin{align*}
F_{n+2^{m-6},2} & - F_{n,2} = \sum_{k=0}^{n+2^{m-6}}(k+2)!\rs{n+2^{m-6}+2}{k+2}_{2}-\sum_{k=0}^{n}(k+2)!\rs{n+2}{k+2}_{2}\\
= & \sum_{k=0}^{2^m-3}(k+2)!\left(\rs{n+2^{m-6}+2}{k+2}_{2}-\rs{n+2}{k+2}_{2}\right)\\
- & \sum_{k=n+2^{m-6}+1}^{2^m-3}(k+2)!\rs{n+2^{m-6}+2}{k+2}_{2}+\sum_{k=n+1}^{2^m-3}(k+2)!\rs{n+2}{k+2}_{2}\\
\equiv & \sum_{k=0}^{2^m-3}\sum_{j=1}^{k+2}(-1)^{k+2-j} \binom{k+2}{j} j^{n+1} (j^{2^{m-6}}-1) (j-1) \hbox{ (mod $2^m$)}
\end{align*}
When $j$ is even, then $j^{n+1}=(2c)^{m+h}$, for some $h \geq 0$, so $2^m \mid j^{n+1}$. Since $m\geq 10$, for odd $j$'s we have
\begin{align*}
\sum_{k=0}^{2^m-3}\sum_{j=1}^{k+2} (-1)^{k+2-j}\binom{k+2}{j}j^{n+1} & (j^{2^{m-6}}-1)(j-1)\\
\equiv 2^{m-4}\sum_{k=0}^{2^m-3}(-1)^{k+1} & \sum_{l=1}^{\lfloor (k+1)/2 \rfloor} \binom{k+2}{2l+1}(2l+1)^{n+1}\\
& \times (l -l^2-4l^3-2l^4) l \hbox{ (mod $2^m$)}.
\end{align*}
The last summation is exactly the $\mathcal S$ and the proof will be similar as above. Combine with the previous case we have the following congruence relation
\begin{eqnarray} \label{14}
F_{n+2^{m-6},2} \equiv F_{n,2} \hbox{ (mod $2^m$), for $m \geq 10$}.
\end{eqnarray}

In the case where $7 \leq m \leq 9$, the remainder of the sum $\sum_{k=0}^{2^m-3}(-1)^{k+1}$\\
$\times \sum_{l=1}^{\lfloor (k+1)/2 \rfloor} \binom{k+2}{2l+1} (2l+1)^{n+1} \left(\sum_{i=1}^{4}l^i2^{i-1}\frac{(2^{m-6}-1)!}{i!(2^{m-6}-i)!}\right) l$ modulo $16$ is computed for $m-1 \leq n \leq m+14$. Divisibility of all these values by $16$ implies that the recent sum is divisible by $16$ therefore
\begin{eqnarray} \label{15}
F_{n+2^{m-6},2} \equiv F_{n,2} \hbox{ (mod $2^m$), for $7 \leq m \leq 9$}.
\end{eqnarray}
Summing up the congruence relations (\ref{14}) and (\ref{15}) gives
\begin{align*}
\omega(A_{2,2^m})=2^{m-6}, \hbox{ for $m \geq 7$}.
\end{align*}
\end{proof}

\bth \label{thmrfubini4,8}
For $m=1$ and $m=2$, the sequence $A_{r,2^m}$ is periodic from the first term and the period length is $\omega({A_{r,2^m}})=1$.
\end{thm}

\begin{proof}
The proof of this theorem is divided into three cases. For  $r=2$ we have
\begin{align*}
F_{n+1,2} & - F_{n,2}= \sum_{k=0}^{n+1}(k+2)!\rs{n+3}{k+2}_{2}-\sum_{k=0}^{n}(k+2)!\rs{n+2}{k+2}_{2}\\
& \equiv  2\left(\rs{n+3}{2}_{2}-\rs{n+2}{2}_{2}\right)+6\left(\rs{n+3}{3}_{2}-\rs{n+2}{3}_{2}\right) \hbox{(mod 4)}\\
&= 2\left(2^{n+1}-2^n\right)+6\left(3^{n+1}-2^{n+1}-(3^n-2^n)\right)\\
&= 2^{n+1}+6(2 \times 3^n-2^n)=4(2^{n-1}+3^{n+1}-3\times 2^{n-1})\\
&= 4(3^{n+1}-2^n) \equiv 0 \hbox{(mod 4)}.
\end{align*}
So we can deduce that $\omega(A_{2,4})=1$ and obviously $\omega(A_{2,2})=1$.

For $r=3$ we can write
\begin{align*}
F_{n+1,3}-F_{n,3}= & \sum_{k=0}^{n+1}(k+3)!\rs{n+4}{k+3}_{3}-\sum_{k=0}^{n}(k+3)!\rs{n+3}{k+3}_{3}\\
\equiv & 6\left(\rs{n+4}{3}_{3}-\rs{n+3}{3}_{3}\right) \hbox{(mod 4)}\\
= & 6\left(3^{n+1}-3^n\right)=6 \times 2 \times 3^n=4 \times 3^{n+1} \equiv 0 \hbox{(mod 4)}.
\end{align*}
Therefore we have $\omega(A_{3,4})=1$ and $\omega(A_{3,2})=1$.

Finally if $r \geq 4$, let $r=4+h$, for some $h \geq 0$, then
\begin{align*}
F_{n+1,r}-F_{n,r}= & \sum_{k=0}^{n+1}(k+r)!\rs{n+1+r}{k+r}_{r}-\sum_{k=0}^{n}(k+r)!\rs{n+r}{k+r}_{r}.
\end{align*}
Since $4 \mid (k+r)!$, for all $k \geq 0$, we can write $F_{n+1,r}-F_{n,r} \equiv 0$ (mod 4). Therefore $\omega(A_{r,4})=1$ and $\omega(A_{r,2})=1$.
\end{proof}

\bth \label{thmrfubini8to64}
If $3 \leq m \leq 6$, after the $(m-1)$th term, the sequence $A_{r,2^m}$ has a period with length $\omega(A_{r,2^m})=2$.
\end{thm}

\begin{proof}
The proof of this theorem is similar to the proof of Theorem \ref{thm2fubini8to64}. Proving for $m=6$ deduces the result for $m=3,4,5$. Let $m=6$, so $n \geq 5$; because $n\geq m-1$. For $3 \leq r \leq 7$ we have
\begin{align*}
F_{n+2,r} - & F_{n,r}= \sum_{k=0}^{n+2}(k+r)!\rs{n+2+r}{k+r}_{r}-\sum_{k=0}^{n}(k+r)!\rs{n+r}{k+r}_{r}\\
\equiv & \sum_{k=0}^{7-r}\sum_{j=r}^{k+r}(-1)^{k+r-j}\binom{k+r}{j} j^{n+1} (j^2-1) \left(\frac{(j-1)!}{(j-r)!}\right) \hbox{ (mod 64)}\\
\equiv & \sum_{k=0}^{7-r}\sum_{j=r}^{k+r}(-1)^{k+r-j}\binom{k+r}{j}j^{n+1} (j^2-1) (j-1) \left(\frac{(j-2)!}{(j-r)!}\right) \hbox{ (mod 64)}
\end{align*}
When $j$ is even, then $j^{n+1}=(2c)^{6+h}$, for some $h \geq 0$, and so $64 \mid j^{n+1}$. For odd $j$'s we have $(j,64)=1$ and the Euler's Theorem gives $j^{32} \equiv 1$ (mod 64). Therefore $j^{n+1+32} \equiv j^{n+1}$ (mod 64) and we can write
\begin{align*}
\sum_{k=0}^{7-r}\sum_{j=r}^{k+r} & (-1)^{k+r-j} \binom{k+r}{j}j^{n+1} (j^2-1) (j-1) \frac{(j-2)!}{(j-r)!} \\
\equiv & \sum_{k=0}^{7-r}\sum_{l=\lfloor r/2 \rfloor}^{\lfloor (k+r-1)/2 \rfloor}(-1)^{k+r-(2l+1)} \binom{k+r}{2l+1} (2l+1)^{n+1} ((2l+1)^2-1)\\
\times & 2l \left(\frac{(2l-1)!}{(2l+1-r)!}\right) \hbox{(mod 64)}\\
\equiv & 16\sum_{k=0}^{7-r}(-1)^{k+r-1}\sum_{l=\lfloor r/2 \rfloor}^{\lfloor (k+r-1)/2 \rfloor} \binom{k+r}{2l+1} (2l+1)^{n+1} \left(\frac{l(l+1)}{2}\right)\\
\times & l \left(\frac{(2l-1)!}{(2l+1-r)!}\right) \hbox{(mod 64)}
\end{align*}
By computation we see that the recent summation is divisible by 4, for $2\leq n \leq 33$. So the proof for $3 \leq r \leq 7$ is completed.

If $r \geq 8$, since $64 \mid 8!$, then $64 \mid (k+r)!$, and
\begin{align*}
F_{n+2,r}-F_{n,r} & =\sum_{k=0}^{n+2}(k+r)!\rs{n+2+r}{k+r}_{r}-\sum_{k=0}^{n}(k+r)!\rs{n+r}{k+r}_{r} \\
& \equiv 0 \hbox{ (mod 64)},
\end{align*}
so $\omega(A_{r,2^6})=2$, for $r \geq 8$, and the proof is completed.
\end{proof}

\bth \label{thmrfubini128}
If $m \geq 7$, after the $(m-1)$th term, the sequence $A_{r,2^m}$ has a period with length $\omega(A_{r,2^m})=2^{m-6}$.
\end{thm}

\begin{proof}
The proof of this theorem is similar to the proof of Theorem \ref{thm2fubini128}.
In the case of $n \geq 2^m-r-1$ and $r \geq 8$ we have
\begin{align*}
F_{n+2^{m-6},r}- F_{n,r}=& \sum_{k=0}^{n+2^{m-6}}(k+r)!\rs{n+2^{m-6}+r}{k+r}_{r}-\sum_{k=0}^{n}(k+r)!\rs{n+r}{k+r}_{r}\\
\equiv  \sum_{k=0}^{2^m-r-1} (k+&r)!\left(\rs{n+2^{m-6}+r}{k+r}_{r}-\rs{n+r}{k+r}_{r}\right) \hbox{ (mod $2^m$)}\\
\equiv \sum_{k=0}^{2^m-r-1} \sum_{j=r}^{k+r}&(-1)^{k+r-j}\binom{k+r}{j} j^{n+1} (j^{2^{m-6}}-1) \left(\frac{(j-1)!}{(j-r)!}\right) \hbox{(mod $2^m$)}
\end{align*}
In the case of $2^m > r > 2^m-m$, since $m\geq 7$ this implies that $r>2^m-m\geq 2^{m-1}$ and so
\beq
2^m \mid (2^{m-1})! \mid (k+r)! \hbox{, for each $k\geq 0$.}
\eeq
Therefore both summations in the above first equation are zero modulo $2^m$ and in this case $\omega(A_{r,2^m})=2^{m-6}$. When $r \leq 2^m -m$, if $j$ is even then $j^{n+1}=(2c)^{2^m-r+h}$, for some $h \geq 0$. So $2^m \mid j^{n+1}$. For odd $j$'s, $(j,2^{m-5})=1$ and by Euler's Theorem $2^{m-5} \mid j^{2^{m-6}}-1$. Since $r \geq 8$ we can write $\frac{(j-1)!}{(j-r)!}= \left(\frac{(j-8)!}{(j-r)!} \right) \prod_{i=1}^{7}  (j-i)$. Therefore $32 \mid \frac{(j-1)!}{(j-r)!}$ and $2^m \mid (j^{2^{m-6}}-1)\left(\frac{(j-1)!}{(j-r)!} \right)$.

In the case of $m-1 \leq n < 2^m-r-1$ and $r \geq 8$ we have
\begin{align*}
F_{n+2^{m-6},r}-F_{n,r} =& \sum_{k=0}^{n+2^{m-6}}(k+r)!\rs{n+2^{m-6}+r}{k+r}_{r}-\sum_{k=0}^{n}(k+r)!\rs{n+r}{k+r}_{r}\\
\equiv & \sum_{k=0}^{2^m-r-1}(k+r)!\left(\rs{n+2^{m-6}+r}{k+r}_{r}-\rs{n+r}{k+r}_{r}\right)\\
& - \sum_{k=n+2^{m-6}+1}^{2^m-r-1}(k+r)!\rs{n+2^{m-6}+r}{k+r} \\
& + \sum_{k=n+1}^{2^m-r-1}(k+r)!\rs{n+r}{k+r} \hbox{ (mod $2^m$)}\\
= & \sum_{k=0}^{2^m-r-1}(k+r)!\left(\rs{n+2^{m-6}+r}{k+r}_{r}-\rs{n+r}{k+r}_{r}\right) + 0,
\end{align*}
and the proof proceeds as the previous case. In the case of $3 \leq r \leq 7$ one can deduce similar to the proof of Theorem \ref{thm2fubini128} that
\begin{align*}
F_{n+2^{m-6},r}-F_{n,r}
\equiv & \sum_{k=0}^{2^m-r-1}\sum_{j=r}^{k+r}(-1)^{k+r-j}\binom{k+r}{j} j^{n+1} (j^{2^{m-6}}-1)\\
\times & \left(\frac{(j-1)!}{(j-r)!}\right) \hbox{ (mod $2^m$)}.
\end{align*}
Exactly the same as Theorem \ref{thm2fubini128}, the even $j$'s run out. Therefore only the terms with odd $j$ remain. So we have
\begin{align*}
 F_{n+2^{m-6},r}-&F_{n,r}\\
\equiv  \sum_{k=0}^{2^m-r-1} & \sum_{l=\lfloor r/2 \rfloor}^{\lfloor (k+r-1)/2 \rfloor}(-1)^{k+r-(2l+1)} \binom{k+r}{2l+1} (2l+1)^{n+1} ((2l+1)^{2^{m-6}}-1)\\
\times & \left(\frac{((2l+1)-1)!}{((2l+1)-r)!}\right) \hbox{ (mod $2^m$)}\\
\equiv 2^{m-5} & \sum_{k=0}^{2^m-r-1}(-1)^{k+r-1} \sum_{l=\lfloor r/2 \rfloor}^{\lfloor (k+r-1)/2 \rfloor} \binom{k+r}{2l+1} (2l+1)^{n+1} \\
\times & \left(\sum_{i=1}^{2^{m-6}}l^i2^{i-1} \left(\frac{(2^{m-6}-1)!}{i!(2^{m-6}-i)!}\right)\right) \left(\frac{(2l)!}{(2l-r+1)!}\right) \hbox{ (mod  $2^m$)}.
\end{align*}
Since $(2l+1,16)=1$, Euler's Theorem shows that $(2l+1)^{n+1+8} \equiv (2l+1)^{n+1}$ (mod 16). If $m \geq 10$, we have
\begin{align*}
F_{n+2^{m-6},r} - & F_{n,r} \equiv 2^{m-4}\sum_{k=0}^{2^m-r-1}(-1)^{k+r-1} \sum_{l=\lfloor r/2 \rfloor}^{\lfloor (k+r-1)/2 \rfloor} \binom{k+r}{2l+1} (2l+1)^{n+1}\\
\times & \left(\frac{l(l+1)}{2}\right) (-2l^2-2l+1) \left(\frac{(2l)!}{(2l-r+1)!}\right) \hbox{ (mod  $2^m$)}.
\end{align*}
Therefore it is sufficient to compute the above summation (without factor $2^{m-4}$) for $3 \leq r \leq 7$ and $9 \leq n \leq 16$ to show that it is divisible by 16.

For $7 \leq m \leq 9$ we evaluate the summation
\begin{align*}
\sum_{k=0}^{2^m-r-1}(-1)^{k+r-1} \sum_{l=\lfloor r/2 \rfloor}^{\lfloor (k+r-1)/2 \rfloor} & \binom{k+r}{2l+1} (2l+1)^{n+1} \sum_{i=1}^{2^{m-6}}l^i2^{i-1}\frac{(2^{m-6}-1)!}{i!(2^{m-6}-i)!}\\
\times & \left(\frac{(2l)!}{(2l-r+1)!}\right)
\end{align*}
for $m-1 \leq n \leq m+6$ to show that it is divisible by 32. Then it follows that $\omega(A_{r,2^m})=2^{m-6}$, for all $m \geq 7$.
\end{proof}

\section{The Conclusion}
We have the final theorem which shows how to compute $\omega(A_{r,s})$ for any $s \in \mathbb{N}$.
\bth
Let $s \in \mathbb{N}$ and $s>1$ with the prime factorization $s=2^{m}p_1^{{m}_1}p_2^{{m}_2} \ldots p_k^{{m}_k}$ and let $D=\{ p_{i}^{m_{i}} \mid p_{i}^{m_{i}}>r, 1\leq i \leq k \}$. Define $E=\{ m_{i}-1 \mid p_{i}^{m_{i}} \in D \}$, $F=\{ \varphi(p_{i}^{m_{i}}) \mid p_{i}^{m_{i}} \in D \}$ and $a=\hbox{max}(E \cup \{m-1\})$ and let $b$ be the Least Common Multiple (LCM) of the elements of $F$. Then
\begin{align}\label{final}
  \omega(A_{r,s}) = \left\{
               \begin{array}{ll}
                 b, & \hbox{ if $0 \leq m \leq 2$ or $2^m \leq r$;}\\
                 LCM(2,b), & \hbox{ if $3 \leq m \leq 6$ and $2^m > r$ ;}\\
                 LCM(2^{m-6},b), & \hbox{ if $m \geq 7$ and $2^m > r$;}
               \end{array}
             \right.
\end{align}
and periodicity of the sequence $A_{r,s}$ is seen after the $a$th term.
\end{thm}

\begin{proof}
Let $l$ be the right hand side of (\ref{final}). For each $d \in D \cup \{2^m\}$, $\omega(A_{r,d}) \mid l$ and for each $p_{j}^{m_{j}} \not\in D$ that $1\leq j \leq k$, $1=\omega(A_{r,p_j^{m_j}}) \mid l$, so
\begin{align*}
F_{n+l,r} &\equiv F_{n,r} \hbox{ (mod $2^m$),}\\
F_{n+l,r} &\equiv F_{n,r} \hbox{ (mod $p_{i}^{m_i}$), for $i=1,2,\ldots,k$.}
\end{align*}
Since $(2^m,p_1^{m_1},p_2^{m_2},\ldots,p_k^{m_k})=1$, the multiplication of all above congruence relations gives the required result.
\end{proof}

\bibliographystyle{amsplain}

\begin{appendix}
\section{Proof of Lemma \ref{appendix}}
Simplify the lemma's relation we have
\begin{eqnarray}\label{factorial}
\frac{2^{i-5}\binom{2^{m-6}-1}{i}}{2^{m-6}-i}=\frac{2^{i-5}(2^{m-6}-1)(2^{m-6}-2)\cdots(2^{m-6}-i+1)}{i!}.
\end{eqnarray}
It is sufficient to show that the right hand side of (\ref{factorial}) is integer. We know that $\binom{2^{m-6}}{i} \in \N$, i.e.,
\begin{eqnarray*}
i! \mid 2^{m-6}(2^{m-6}-1)\cdots (2^{m-6}-i+1).
\end{eqnarray*}
If $O_i$ denotes the product of the odd factors of $i!$, since $(O_i,2^{m-6})=1$, then $O_i \mid (2^{m-6}-1)\cdots (2^{m-6}-i+1)$. So in (\ref{factorial}) we only want to prove that the powers of 2 in $2^{i-5}(2^{m-6}-1)(2^{m-6}-2)\cdots(2^{m-6}-i+1)$ is greater than or equal to the powers of 2 in $i!$.
Let $A$ and $B$ be the greatest integers such that
\begin{align*}
2^A & \mid (2^{m-6}-1)(2^{m-6}-2)\cdots(2^{m-6}-i+1), \hbox{ and }\\
2^B & \mid i!.
\end{align*}
Let $e$ be the unique integer such that $2^e\leq i < 2^{e+1}$. So
\begin{align}\label{def}
A=\sum_{k=1}^{e} \lfloor \frac{i-1}{2^k} \rfloor,  B=\sum_{k=1}^{e} \lfloor \frac{i}{2^k} \rfloor.
\end{align}
If we show that
\begin{eqnarray}\label{claim}
B-A\leq e
\end{eqnarray}
then the lemma is concluded if it is proved that
\begin{eqnarray}\label{conclusion}
i+A \geq B+5.
\end{eqnarray}
Determine (\ref{conclusion}) by induction. For $e=2$, three cases are exist as shown in the following table.

\makebox[12cm]{
\begin{tabular}[b]{c|c|c}
  $i$ & $A$ & $B$\\
\cline{1-3}
  5 & 3 & 3 \\
\cline{1-3}
  6 & 3 & 4 \\
\cline{1-3}
  7 & 4 & 4 \\
\end{tabular}
}

\noindent For $e \geq 3$ one can deduce by simple induction that
\begin{eqnarray}
2^e\geq e+5,
\end{eqnarray}
so $i\geq 2^e \geq e+5$. Add $B-e$ to these inequalities and use (\ref{claim}) demonstrates (\ref{conclusion}) for $i\geq 8$.

To prove the equality part of (\ref{claim}), consider $i=2^e$, for $e \geq 3$. In the case where $2^e < i < 2^{e+1}$, consider $1 \leq k < e$. By The Division Algorithm, $i=2^{k}h+r$ where $0\leq r<2^k$,  therefore $\lfloor \frac{i}{2^k} \rfloor=\lfloor h+\frac{r}{2^k} \rfloor=h$. Furthermore we have
\begin{align}\label{div_alg}
\lfloor \frac{i-1}{2^k} \rfloor=\lfloor h+\frac{r-1}{2^k} \rfloor=h+\lfloor \frac{r-1}{2^k} \rfloor.
\end{align}
Dividing the inequality $-1 \leq r-1< 2^k-1$ by $2^k$ gives  $-1<\frac{-1}{2^k} \leq \frac{r-1}{2^k}< \frac{2^k-1}{2^k}<1$ and therefore $-1 \leq \lfloor \frac{r-1}{2^k} \rfloor \leq 0$. It means that the right hand side of (\ref{div_alg}) is equal to $h$ or $h-1$. Therefore
\begin{align}\label{difference}
0 \leq \lfloor \frac{i}{2^k} \rfloor-\lfloor \frac{i-1}{2^k} \rfloor \leq 1, \hbox{ for $1 \leq k < e$.}
\end{align}
If $k=e$, we have $i=2^e+r'$, $r'>0$, and $i-1=2^e+r'-1$. So $\lfloor \frac{i}{2^e} \rfloor=\lfloor 1+\frac{r'}{2^e} \rfloor=1$. Since $0 \leq r'-1<2^e$, then $0 \leq \frac{r'-1}{2^e}<1$ and so $\lfloor \frac{r'-1}{2^e} \rfloor=0$. Now we can deduce that $\lfloor \frac{i-1}{2^e} \rfloor=\lfloor 1+\frac{r'-1}{2^e} \rfloor=1$. Using (\ref{difference}) and $\lfloor \frac{i}{2^e} \rfloor-\lfloor \frac{i-1}{2^e} \rfloor=0$ we have
\begin{align}
\sum_{k=1}^{e}\lfloor \frac{i}{2^k} \rfloor-\lfloor \frac{i-1}{2^k} \rfloor \leq e-1 < e.
\end{align}
Substituting these results in the definitions of $A$ and $B$ in (\ref{def}) gives (\ref{claim}).

\end{appendix}
\vskip 1.5cm
\end{document}